\documentclass[11pt]{article}
\usepackage{epic,latexsym,amssymb}
\usepackage{color}
\usepackage{tikz}
\usepackage{amsmath,amsthm}
\usepackage{dirtytalk}

\newtheorem{thm}{Theorem}
\newtheorem{lem}{Lemma}

\newtheorem{claim}{Claim}

\newtheorem{cor}{Corollary}

\newtheorem{defn}{Definition}

\newcommand{\Z}{{\Z B}}

\let\oldenumerate\enumerate
\renewcommand{\enumerate}{
  \oldenumerate
  \setlength{\itemsep}{0pt}
  \setlength{\parskip}{0pt}
  \setlength{\parsep}{0pt}
}

\def\vertex(#1){\put(#1){\circle*{2}}}
\def\vertexo(#1){\put(#1){\circle{2}}}
\def\vert(#1){\put(#1){\circle*{1.5}}}
\def\verto(#1){\put(#1){\circle{1.5}}}
\def\lab(#1)#2{\put(#1){\makebox(0,0)[c]{#2}}}

\begin{document}

\title{ An improved lower bound on the number\\ of $1$-nearly independent vertex subsets}

\author{Zekhaya B. Shozi \thanks{Research supported by University of KwaZulu-Natal.}\\
	School of Mathematics, Statistics \& Computer Science\\
	University of KwaZulu-Natal\\
	Durban, 4000 South Africa\\
\small \tt Email: zekhaya@aims.ac.za
}

\date{}
\maketitle

\begin{abstract}
Let $G=(V(G),E(G))$ be a graph with set of vertices $V(G)$ and set of edges $E(G)$. For $k\ge 0$ an integer, a subset $I_k$ of $V(G)$ is called a $k$-nearly independent vertex subset of $G$ if $I_k$ induces a subgraph of size $k$ in $G$. The number of such subsets in $G$ is denoted by $\sigma_k(G)$. In this paper we continue the study of $\sigma_1$. In particular, we prove the lower bound on $\sigma_1$ for a connected graph that contains a cycle and also characterise the two extremal graphs. This improves the result obtained in [E. O. D. Andriantiana and Z. B. Shozi. The number of 1-nearly independent vertex
subsets. \textit{Quaestiones Mathematicae}, accepted].
 \end{abstract}

\newpage
	
\section{Introduction}

A simple and undirected graph $G$ is an ordered pair of sets $(V(G),E(G))$, where $V(G)$ is a nonempty set of elements which are called \emph{vertices} and $E(G)$ is a (possibly empty) set of $2$-element subsets of $V(G)$ which are called \emph{edges}. The number of vertices of a graph $G$ is the \emph{order} of $G$ and is denoted by $n$, while the number of edges of $G$ is a the \emph{size} of $G$ and is denoted by $m$. For convenience, we often write $uv$ instead of $\{u,v\}$ to represent the edge joining the vertices $u$ and $v$ in a graph $G$.

A subset $I_0$ of $V(G)$ is called an \emph{independent (vertex) subset} of $G$ if no two vertices in $I_0$ are adjacent in $G$. The number of independent  vertex subsets of a graph has been well-studied in the literature. See the survey in \cite{wagner2010maxima}, where it is called the \emph{Merrifield-Simmons index}. Also the book \cite{li2012shi} contains an extensive survey around this topic. Merrifield and Simmons \cite{Merrifield198055} used the number independent  vertex subsets  of molecular graphs and as a measure of  molecular complexity, bond strength and boiling point  of the associated molecules. This ignited the interest of both chemists and mathematicians to study this graph invariant. The results obtained exceed the class of molecular graphs. In \cite{Andriantiana2013724} the structure of the tree with a given degree sequence $D$ that has the largest number of independent  vertex subsets is fully characterised. The result implies, as corollaries, characterisations of trees with largest number of independent  vertex subsets in various other classes like trees with fixed order, or with fixed order and given maximum degree.

Various ways of generalisation of the notion of independent  vertex subsets have been attempted. For example \cite{jagota2001generalization} generalised the concept of maximal-independent set, by considering the $k$-insulated set $S$ of a graph $G$ defined as a subset of its vertices such that each vertex in $S$ is adjacent to at most $k$ other vertices in $S$ and each vertex not in $S$ is adjacent to at least $k+1$ vertices in $S$. See also \cite{drmota1991generalized}, which studies subsets which do not contain pair of vertices with distance shorter than a specified integer $k$. 

Andriantiana and Shozi \cite{andriantiana2024number} proposed a new other generalisation. Firstly, they defined a \emph{$k$-nearly independent vertex subset} as a subset $I_k$ of $V(G)$ such that $I_k$ contains exactly $k$ pairs of adjacent vertices of $G$. They then denoted the number of $k$-nearly independent vertex subsets of $G$ by $\sigma_k(G)$. Remarkably, $\sigma_0(G)$ is the number of independent  vertex subsets of $G$. In their paper, they studied $\sigma_1$, where they established that the lower bound on $\sigma_1$ is uniquely attained by the star $K_{1,n-1}$. Since the star $K_{1,n-1}$ is a tree, it does not contain a cycle. It is, therefore, a natural question to find the lower bound on $\sigma_1$ if the given graph contains a cycle.

The rest of the paper is structured as follows. Section \ref{preliminary} is a preliminary section that contains some technical formulas that will be used in this paper. We will also include some known results of $\sigma_1$ in this section. Our main result is in Section \ref{Sec:Main}, where we study all graphs of order $n$ and size $m$ that contain at least one cycle and characterise those that achieve the minimum $\sigma_1$.

\section{Preliminary}
\label{preliminary}

For graph theory notation and terminology, we generally follow~\cite{henning2013total}. Let $G$ be a graph with vertex set $V(G)$, edge set $E(G)$, order $n = |V(G)|$ and size $m = |E(G)|$. We denote the degree of a vertex $v$ in $G$ by $\deg_G(v)$. The minimum and maximum degree of $G$ is denoted by $\delta(G)$ and $\Delta(G), respectively.$ A vertex of degree~$q$ in $G$ whose neighbors have degrees~$d_1, d_2, \ldots, d_q$, respectively, where $d_1 \le d_2 \le \cdots \le d_q$ is called a $(d_1,d_2,\ldots,d_q)$-\emph{vertex}.

For positive integers $r$ and $s$, we denote by $K_{r,s}$ the complete bipartite graph with partite sets $X$ and $Y$ such that $|X|=r$ and $|Y|=s$. A complete bipartite graph $K_{1,n-1}$ is also called a \emph{star} in the literature. For a subset $S$ of $V(G)$, we define $G-S$ as the subgraph obtained from $G$ by deleting all the vertices in $S$ along with their incident edges. If $S = \{u\}$, we write $G - u$ instead of $G - \{u\}$.  The complement of a graph $G$ is denoted by $\overline{G}$. A path, a cycle and a complete graph on $n$ vertices is denoted by $P_n$, $C_n$ and $K_n$, respectively.

\subsection{Recursive formula}

For any vertex $v$ of a graph $G$,
\begin{align*}
\label{Eq:Rec}
\sigma_1(G)
=\sigma_1(G-v)+\sigma_1(G-N[v])+\sum_{u\in N(v)}\sigma_0(G-(N[u]\cup N[v])),
\end{align*}
where $\sigma_1(G-v)$ counts the number of $1$-nearly independent  vertex subsets that do not contain $v$, $\sigma_1(G-N[v])$ counts those that contain $v$ as a single vertex, and $\sum_{u \in N(v)}\sigma_0(G-(N[u]\cup N[v]))$ counts those that contain $v$ is as an edge.

\begin{defn}
\label{Def:Good}
    Let $G = (V(G), E(G))$ be a graph with vertex set $V(G)$ and edge set $E(G)$. If $e=uv \in E(G)$, then $e$ is a good edge if $N_G[u]\cup N_G[v] = V(G)$. The graph $G$ is a good graph if for every edge $e \in E(G)$, $e$ is a good edge. Let
    $$\mathcal{H} = \{ G \mid G \text{ is a good graph} \}.$$
\end{defn}
It follows from the definition that a good graph has to be connected.

Andriantiana and Shozi \cite{andriantiana2024number} established the following result.
\begin{thm}[\cite{andriantiana2024number}]
\label{sigma-1-of-G-at-least-m-with-equality-iff-G-is-a-good-graph}
    If $G$ is a connected graph of size $m$, then
    \begin{align*}
        \sigma_1(G) \ge m,
    \end{align*}
    with equality if and only if $G \in \mathcal{H}$.
\end{thm}

Since the star $K_{1,n-1}$ is the element of $\mathcal{H}$ with fewest edges among all elements of $\mathcal{H}$ with $n$ vertices, we obtain the following corollary.

\begin{cor}[\cite{andriantiana2024number}]
    \label{among-all-connected-graphs-of-order-n-the-star-has-the-smallest-sigma-1}
    If $G$ is a connected graph of order $n$, then 
    \begin{align*}
        \sigma_1(G) \ge n-1,
    \end{align*}
    with equality if and only if $G \cong K_{1, n-1}$.
\end{cor}

\section{Main result}
\label{Sec:Main}

In this section we provide a characterisation of the structure of a graph that contains a cycle and has the minimum $\sigma_1$. We first prove the following important lemmas that will be useful in proving our main result.
 \begin{lem}
    \label{G-good-graph-with-a-cycle-then-G-cannot-contain-a-bridge}
        If $G$ is a good graph and $G$ contains a cycle, then $G$ does not contain a bridge.
    \end{lem}
    \begin{proof}
        Let $G$ be a good graph that contains a cycle, and suppose that $G$ contains a bridge $e=uv$. Then $G-e$ is disconnected and has two components, namely the component $G_u$ containing $u$ and the component $G_v$ containing $v$.  Let $N_G(u) \setminus \{v\} = \{u_1, u_2, \ldots, u_{\deg_G(u)-1}\}$ and $N_G(v) \setminus \{u\} = \{v_1, v_2, \ldots, v_{\deg_G(v)-1}\}$, where possibly $N_G(v) \setminus \{u\} = \emptyset$. If $w_1 \in V(G_u)$ and $w_2 \in V(G_v)$, then every $w_1$-$w_2$ path in $G$ traverses the edge $e$, otherwise $e$ would not be a bridge. If $C$ is a cycle of $G$, then $C$ lies entirely on $G_u$ or entirely on $G_v$. Without any loss of generality, suppose that $C$ lies on $G_u$, and let $xy$ be an edge of $C$ such that $x\ne u\ne y$. Thus, in addition to the $1$-nearly independent  vertex subsets made of the endpoints of $xy$, the set $\{x,y,z\}$ is also a 1-nearly independent vertex subset, for any $z\in V(G_v)$, implying that $xy$ is not a good edge. However, this contradicts the fact that $G$ is a good graph.
    \end{proof}

\begin{lem}
    \label{G-good-graph-with-a-cycle-then-G-cannot-contain-a-cut-vertex}
        If $G$ is a good graph and $G$ contains a cycle, then $G$ does not contain a cut-vertex.
\end{lem}
\begin{proof}
     Let $G$ be a good graph that contains a cycle, and suppose that $G$ contains a cut-vertex $v$. Then $G-v$ is disconnected and has at least two components, and there exists a pair of vertices $w$ and $z$ such that $v$ lies on every $w$-$z$ path. Let $P: w=w_1, w_2, \ldots, w_k=z$ be a $w$-$z$ path in $G$. Then $v=w_i$ for some $i$, where $2\le i \le k-1$. Let $C$ be a cycle of $G$. Then we observe that $V(P)\setminus V(C) \ne \emptyset$. Let $w_1 \in V(P)\setminus V(C)$ and let $xy$ be an edge of $C$ such that $x\ne v\ne y$. Thus, in addition to the $1$-nearly independent  vertex subsets made of the endpoints of $xy$, the set $\{x,y,w_1\}$ is also a 1-nearly independent vertex subset, implying that $xy$ is not a good edge. However, this contradicts the fact that $G$ is a good graph.
\end{proof}

We are now in a good position to present the proof of our main result.

\begin{thm}
\label{thm:main}
    Let $G$ be a connected graph of order $n\ge 3$. If $G$ contains a cycle, then 
    \begin{align*}
        \sigma_1(G) \ge \begin{cases}
            n & \text{ if } n =3\\
            2n-4 & \text{ if } n \ge 4,
        \end{cases}
    \end{align*}
    with equality if and only if $G \in \{K_3, K_{2,n-2}\}$.
\end{thm}

\begin{proof}
    Let $G$ be a connected graph of order $n\ge 3$ that contains a cycle.  By Lemma \ref{G-good-graph-with-a-cycle-then-G-cannot-contain-a-bridge}, $G$ does not contain a bridge. Thus, every edge of $G$ is a cycle edge, implying that $\delta(G) \ge 2$. If $n=3$, then $G \cong K_3$ and $\sigma_1(G) = 3 =n$. Hence, we may assume that $n\ge 4$, for otherwise there is nothing left to prove. If $n=4$, then $\sigma_1(K_4)=6>\sigma_1(K_4-e)=5>\sigma_1(K_{2,2})=4=2n-4$, thereby proving the base case. Assume the result is true for all graphs of order $n<k$, where $k\ge 4$, and let $G$ be a graph of order $n=k$. By Theorem \ref{sigma-1-of-G-at-least-m-with-equality-iff-G-is-a-good-graph}, $\sigma_1(G) \ge |E(G)|\ge n-1$, where the $(n-1)$-bound is uniquely attained by the star $K_{1,n-1}$, and the $|E(G)|$-bound is attained if $G$ is a good graph. We emphasize that the star $K_{1,n-1}$ is also a good graph. Thus, it is clear that we want to find the minimum number of edges in a good graph $G$ of order $n$ that contains a cycle. We proceed with the following series of claims.

\begin{claim}
\label{if-max-degree-n-1-then-sigm_1-G-strictly-greater-than-2n-4}
    If $\Delta(G)=n-1$, then $\sigma_1(G) > 2n-4$.
\end{claim}
\begin{proof}
    Suppose $\Delta(G)=n-1$, and let $v$ be a vertex of degree $n-1$ in $G$. Then the graph $G' = G-v$ has order $n-1$ and, by Lemma \ref{G-good-graph-with-a-cycle-then-G-cannot-contain-a-cut-vertex}, $G$ does not contain a cut-vertex, so $G-v$ is connected. Therefore, we have
    \begin{align*}
        \sigma_1(G) &= \sigma_1(G-v) + \sigma_1(G - N_G[v]) + \displaystyle \sum\limits_{u\in N_G(v)}\sigma_0(G-(N_G[u]\cup N_G[v]))\\
        &\ge 2(n-1)-4 + \sigma_1(\emptyset) + (n-1)\sigma_0(\emptyset)\\
        &=(2n-4) + (n-3)\\
        &>2n-4, & \text{since } n\ge 4.
    \end{align*}
    This completes the proof of Claim \ref{if-max-degree-n-1-then-sigm_1-G-strictly-greater-than-2n-4}.
\end{proof}

By Claim \ref{if-max-degree-n-1-then-sigm_1-G-strictly-greater-than-2n-4}, we may assume that $\Delta(G) \le n-2$, for otherwise that is nothing left to prove. Also, recall that $G$ contains neither a bridge nor a cut-vertex.

\begin{claim}
\label{if-min-degree-atleast-3-then-sigm_1-G-strictly-greater-than-2n-4}
    If $\delta(G)\ge 3$, then $\sigma_1(G) > 2n-4$.
\end{claim}
\begin{proof}
    Suppose $\delta(G) = q \ge 3$, and let $v$ be a vertex of degree $q$ in $G$. Then the graph $G' = G-v$ has order $n-1$ and, by Lemma \ref{G-good-graph-with-a-cycle-then-G-cannot-contain-a-cut-vertex}, $G$ does not contain a cut-vertex, so $G-v$ is connected. Therefore, we have
    \begin{align*}
        \sigma_1(G) &= \sigma_1(G-v) + \sigma_1(G - N_G[v]) + \displaystyle \sum\limits_{u\in N_G(v)}\sigma_0(G-(N_G[u]\cup N_G[v]))\\
        &\ge 2(n-1)-4 + \sigma_1(\emptyset) + q\sigma_0(\emptyset)\\
        &\ge (2n-4) -2 + 0 + 3(1)\\
        &>2n-4.
    \end{align*}
    This completes the proof of Claim \ref{if-min-degree-atleast-3-then-sigm_1-G-strictly-greater-than-2n-4}.
\end{proof}

By Claim \ref{if-min-degree-atleast-3-then-sigm_1-G-strictly-greater-than-2n-4}, we may assume that $\delta(G) = 2$, for otherwise that is nothing left to prove.

\begin{claim}
\label{if-G-good-graph-n-min-deg-2-then-max-deg-at-least-n-2}
    If $G$ is a good graph of order $n$ with $\delta(G)=2$, then $\Delta(G) \ge n-2$.
\end{claim}
   
\begin{proof}
    Let $G$ be a good graph of order $n$ with $\delta(G)=2$, and let $u$ be a vertex of $G$ with $\deg_G(u) =2$. Suppose, to the contrary, that $\Delta(G)\le n-3$, and let $v$ be a vertex of maximum degree in $G$. Denote by $\overline{N_G(v)}$ the set of vertices of $G$ that are not adjacent of $v$ in $G$, and $N_G(v)$ the set of vertices of $G$ that are adjacent to $v$ in $G$. Clearly, $\overline{N_G(v)} \cap N_G(v) =\emptyset$. Furthermore, since $\deg_G(v) \le n-3$, we must have $|\overline{N_G(v)}| \ge 2$, implying that there are at least two vertices of $G$ that are not adjacent to $v$ in $G$. Let $u$ be a vertex of degree $2$ in $G$. Then either $u\in N_G(v)$ or $u\in \overline{N_G(v)}$. 
    
    Suppose $u \in N_G(v)$. Since $\deg_G(u) =2$, there is at most one neighbour of $u$ that belongs to $\overline{N_G(v)}$. However, $|\overline{N_G(v)}|\ge 2$, so there is at least one vertex $\overline{v} \in \overline{N_G(v)}$ such that $\overline{v}$ is adjacent to neither $u$ nor $v$. Thus, the edge $uv$ is not a good edge in $G$, contradicting the fact that $G$ is a good graph. Hence, we may assume that $u \in \overline{N_G(v)}$. Since $\deg_G(u) =2$ and $\deg_G(v)\le n-3$, there must exist an edge $vw \in E(G)$ such that $w\notin N_G(u)$. Thus, the edge $vw$ is not a good edge in $G$, contradicting the fact that $G$ is a good graph.
\end{proof} 

 By Claim \ref{if-G-good-graph-n-min-deg-2-then-max-deg-at-least-n-2}, if $G$ is a good graph of order $n$ with $\delta(G)=2$, then $\Delta(G) \ge n-2$. Also, recall that by Claim \ref{if-max-degree-n-1-then-sigm_1-G-strictly-greater-than-2n-4}, we have  $\Delta(G) \le n-2$. Consequently, $\Delta(G)=n-2$.

\begin{claim}
    \label{every-vertex-of-deg-2-in-G-is-an-n-minus-2-n-minus-2-vertex}
    If $G$ is a good graph of order $n$ with $\delta(G) =2$ and $\Delta(G)=n-2$, then every vertex of degree $2$ in $G$ is an $(n-2, n-2)$-vertex.
\end{claim}

\begin{proof}
    Let $G$ be a good graph of order $n$ with $\delta(G) =2$ and $\Delta(G)=n-2$, and let $u_1$ be a vertex of degree $2$ in $G$ whose neighbours are $v$ and $w$, respectively. Suppose, to the contrary, that $\deg_G(v) = q \le n-2$ and $\deg_G(w) = r \le n-3$. Without any loss of generality, we may assume that $q>r$. Let 
    \begin{align*}
        N_G(v) =\displaystyle \bigcup\limits_{i=1}^q \{u_i\} \text{ and } N_G(w) =\displaystyle \bigcup\limits_{j=1}^r \{w_j\}.
    \end{align*}
    Note that it is possible that $N_G(w) \subset N_G(v)$. Since $\deg_G(u_1) =2$ and $N_G(u_1) =\{v,w\}$, the edge $u_1u_i$, where $2\le i \le q$, does not exist in $G$. Moreover, since $q>r$, there exists a vertex $x=u_i$, where $2\le i \le q$, such that $x$ is a neighbour of $v$ and $x$ is not a neighbour of $w$. Thus, the edge $u_1w$ is not a good edge in $G$, contradicting the fact that $G$ is a good graph.
\end{proof}

\begin{claim}
    \label{every-vertex-of-deg-n-minus-2-in-G-is-a-2-2-2-2-vertex}
    If $G$ is a good graph of order $n$ with $\delta(G) =2$ and $\Delta(G)=n-2$, then every vertex of degree $n-2$ in $G$ is a $(\underbrace{2,2, \ldots, 2}_{n-2 \text{ times }})$-vertex.
\end{claim}

\begin{proof}
    Let $G$ be a good graph of order $n$ with $\delta(G) =2$ and $\Delta(G)=n-2$. Let $u$ be a vertex of degree $n-2$ in $G$, and whose set of neighbours is $N_G(u) = \{u_1, u_2, \ldots, u_{n-2}\}$. Suppose, to the contrary, that there exists a vertex $u_i\in N_G(u)$, where $1\le i\le n-2$, such that $\deg_G(u_i) \ge 3$. 

    Rename the vertices $u_i \in N_G(u)$, for $1\le i\le n-2$, according to the descending order of their degrees. That is, $\deg_G(u_1) \ge \deg_G(u_2) \ge \cdots \ge \deg_G(u_{n-2})$. Since $\Delta(G)=n-2$, each vertex $u_i\in N_G(u)$ is adjacent to at most $n-4$ other vertices in $N_G(u)$, implying that the edge $uu_{n-2}$ does not exist in $G$. Let $w$ be a vertex of degree $2$ in $G$, and whose neighbours are $w_1$ and $w_2$, respectively. By Claim \ref{every-vertex-of-deg-2-in-G-is-an-n-minus-2-n-minus-2-vertex}, $\deg_G(w_i) = n-2$ for $1 \le i\le  2$.

    Suppose $w\notin N_G(u)$. Since $\deg_G(u) = n-2$, both the neighbours of $w$ must belong to $N_G(u)$. Thus, we may assume that $w_i = u_i$ for $1 \le i\le  2$. However, the edge $uu_3$ is not a good edge in $G$ because both its endpoints are not adjacent to $w$. Hence, we may assume that $w\in N_G(u)$. Without any loss of generality, assume that $w_1 = u$. If $w_2\in N_G(u)$, then $\deg_G(u)=\deg_G(w_1) = n-1$, a contradiction. Hence, we may assume that $w_2 \notin N_G(u)$. However, the edge $u_iu_j$, for $1 \le i\le n-3$ and $i+1\le j\le n-3$, is not a good edge in $G$ because both its endpoints are not adjacent to $w$. This completes the proof of Claim \ref{every-vertex-of-deg-n-minus-2-in-G-is-a-2-2-2-2-vertex}.
\end{proof}

\begin{claim}
    \label{every-vertex-either-has-degree-2-or-has-degree-n-minus-2}
    If $G$ is a good graph of order $n$ with $\delta(G)=2$ and $\Delta(G)=n-2$, then every vertex of $G$ either has degree $2$ in $G$ or has degree $n-2$ in $G$.
\end{claim}

\begin{proof}
    Let $G$ be a good graph of order $n$ with $\delta(G)=2$ and $\Delta(G)=n-2$, and let $v$ be a vertex of $G$ such that $\deg_G(v)=q\ge 3$. We will show that $\deg_G(v)=n-2$. Suppose, to the contrary, that $\deg_G(v)\le n-3$. Let $u$ be a vertex of degree $n-2$ in $G$, and whose set of neighbours is $N_G(u) = \{u_1, u_2, \ldots, u_{n-2}\}$. Furthermore, let $N_G(v) =\{v_1, v_2, \ldots, v_q\}$ be the set of neighbours of $v$.

    Since $G$ has $n$ vertices, and $\deg_G(u)=n-2$, either $v\in N_G(u)$ or $v\notin N_G(u)$. Suppose $v\in N_G(u)$. By Claim \ref{every-vertex-of-deg-n-minus-2-in-G-is-a-2-2-2-2-vertex}, $\deg_G(u_i)=2$ for all $i$, where $1\le i \le n-2$, a contradiction. Hence, we may assume that $v \notin N_G(u)$. If $v\notin N_G(u)$, then we must have $N_G(v) \subset N_G(u)$. Thus, there exits $k$, where $1\le k\le n-2$, such that $uu_k \in E(G)$ and $vu_k \notin E(G)$. Such an edge $uu_k$ is not a good edge because both its endpoints are not adjacent to $v$. Thus, we have a contradiction and so we conclude that $\deg_G(v)=2$ or $\deg_G(v)=n-2$. 
\end{proof}

Recall that by Claim \ref{if-min-degree-atleast-3-then-sigm_1-G-strictly-greater-than-2n-4}, $\delta(G)=2$. By Claim \ref{if-max-degree-n-1-then-sigm_1-G-strictly-greater-than-2n-4}, we have  $\Delta(G) \le n-2$. Furthermore, by Claim \ref{if-G-good-graph-n-min-deg-2-then-max-deg-at-least-n-2}, if $G$ is a good graph of order $n$ with $\delta(G)=2$, then $\Delta(G) \ge n-2$, implying that $\Delta(G)=n-2$. By Claim \ref{every-vertex-of-deg-2-in-G-is-an-n-minus-2-n-minus-2-vertex}, every vertex of degree $2$ in $G$ is an $(n-2, n-2)$-vertex. By Claim \ref{every-vertex-of-deg-n-minus-2-in-G-is-a-2-2-2-2-vertex}, every vertex of degree $n-2$ in $G$ is a $(\underbrace{2,2, \ldots, 2}_{n-2 \text{ times }})$-vertex. Lastly, by Claim \ref{every-vertex-either-has-degree-2-or-has-degree-n-minus-2}, every vertex of $G$ either has degree $2$ in $G$ or has degree $n-2$ in $G$. These properties of the connected graph $G$ of order $n\ge 4$ that contains a cycle and has minimum $\sigma_1$ are sufficient to deduce that $G \cong K_{2,n-2}$, thereby completing the proof of Theorem \ref{thm:main}.
\end{proof}

\newpage
\bibliographystyle{abbrv} 
\bibliography{references}

\begin{thebibliography}{1}

\bibitem{andriantiana2024number}
E.~O. Andriantiana and Z.~B. Shozi.
\newblock The number of 1-nearly independent vertex subsets.
\newblock {\em Quaestiones Mathematicae}, pages 1--21, 2024.

\bibitem{Andriantiana2013724}
E.~O.~D. Andriantiana.
\newblock Energy, hosoya index and merrifield-simmons index of trees with
  prescribed degree sequence.
\newblock {\em Discrete Applied Mathematics}, 161(6):724 – 741, 2013.

\bibitem{drmota1991generalized}
M.~Drmota and P.~Kirschenhofer.
\newblock On generalized independent subsets of trees.
\newblock {\em Random Structures \& Algorithms}, 2(2):187--208, 1991.

\bibitem{henning2013total}
M.~A. Henning and A.~Yeo.
\newblock {\em Total domination in graphs}.
\newblock Springer, 2013.

\bibitem{jagota2001generalization}
A.~Jagota, G.~Narasimhan, and L.~{\v{S}}olt{\'e}s.
\newblock A generalization of maximal independent sets.
\newblock {\em Discrete applied mathematics}, 109(3):223--235, 2001.

\bibitem{li2012shi}
X.~Li.
\newblock Y, shi, i. gutman, graph energy, 2012.

\bibitem{Merrifield198055}
R.~E. Merrifield and H.~E. Simmons.
\newblock The structures of molecular topological spaces.
\newblock 55(1):55 – 75, 1980.

\bibitem{wagner2010maxima}
S.~Wagner and I.~Gutman.
\newblock Maxima and minima of the hosoya index and the merrifield-simmons
  index: a survey of results and techniques.
\newblock {\em Acta Applicandae Mathematicae}, 112:323--346, 2010.

\end{thebibliography}

\end{document}